\documentclass[12pt]{amsart}
\usepackage{amsmath, amsfonts, amssymb, amsthm}
\usepackage{mathtools}
\usepackage{graphicx}
\usepackage{caption}
\usepackage{subcaption}

\title{Consonance in music - the Pythagorean approach revisited}

\author{Jan Cichowlas}
\address{Faculty of Mathematics and Information Science, Warsaw University of Technology, 00-661 Warsaw, Poland}
\author{Paweł Dłotko}
\address{Dioscuri Centre in Topological Data Analysis, Mathematical Institute, Polish Academy of Sciences, 00-656 Warsaw, Poland }
\author{Marek Kuś}
\address{Center for Theoretical Physics, Polish Academy of Sciences, 02-668 Warsaw, Poland}
\author{Jan Spaliński}
\address{Faculty of Mathematics and Information Science, Warsaw University of Technology, 00-661 Warsaw, Poland}

\newtheorem{lemma}{Lemma}

\newtheorem{thm}{Theorem}
\newtheorem{corollary}{Corollary}

\newtheorem{example}{Example}

\def\cosim{\textrm{Cosim\,}}

\def\cons{\textrm{Cons}\,}
\def\O{\textbf{O}\,}

\begin{document}

\begin{abstract} The Pythagorean school attributed consonance in music
to simplicity of frequency ratios between musical tones. In the last two centuries,
the consonance curves developed by Helmholtz, Plompt and Levelt shifted focus to psycho-acoustic considerations in perceiving consonances. The appearance of peaks of these curves at the ratios considered by the Pythagorean school, and which were a consequence of an attempt to understand the world by nice mathematical proportions,  remained a curiosity. 
This paper addresses this curiosity, by describing a mathematical model of musical sound, along with a mathematical definition of consonance. First, we define pure, complex and mixed tones as mathematical models of musical sound. 
By a sequence of numerical experiments and analytic calculations, we show that  {\sl continuous cosine similarity}, abbreviated as cosim,  applied to these models quantifies the elusive concept of consonance as a frequency ratio which gives a local maximum of the cosim function. We prove that these maxima occur at the ratios considered as consonant in classical music theory. Moreover, we provide a simple explanation why the number of musical intervals  considered as consonant by musicians is finite, but has been increasing over the centuries. Specifically, our  formulas show that the number of consonant intervals changes with the depth of the tone (the number of harmonics present).  
\end{abstract}

\keywords{Consonance, Pure Tone, Complex Tone, Cosine Similarity, Consonance Curve.}

\maketitle
\section{Introduction}

The notion of consonance has been a fundamental concept in the study of music since antiquity. The Pythagorean School defined consonant musical intervals as those which result from vibrating string lengths (of same tension) adhered to simple mathematical proportions, known as \emph{Pythagorean Ratios}, involving the numbers one to four, with particular emphasis on the perfect fifth (3:2) and the perfect fourth (4:3).

In this paper, we revisit the Pythagorean foundations of musical consonance  (and dissonance - the lack of consonance). To achieve this, we begin with precise mathematical definitions of a \emph{pure tone} (a single sine wave) and a \emph{complex tone} (a pure tone with a finite number of successive harmonics of varying amplitudes). Applying cosine similarity \footnote{Loosely speaking, this is an uncentered version of the well known Pearson coefficient from statistics, applied to a sample of a deterministic musical tone.} and its continuous analogue (abbreviated `Cosim') to a pair of pure (or complex) tones of frequencies $f$ and $g$, we 
define the function $\cons(f,g)$ which gives a measure of consonance of the two tones (the larger the value of the function, the more consonant the tones $f$ and $g$ are).
Next, letting $r$ be a ratio in the interval $[1,2]$, we consider $\cons(f,rf)$, as a a function of $r$. 

The main contribution of this paper is summarized in Theorem \ref{maintheorem}, where we show that the local maxima of $\cons(f,rf)$ occur precisely at the frequency ratios identified by the Pythagorean school and reflected in the Plomp and Levelt consonance curve~\cite{Plomp65} (obtained as a result of a number of psychoacoustic experiments on a cohort of people). Additionally, we establish a connection between the number of these maxima and the \emph{depth} of a complex tone, understood as the number of harmonics present. More specifically, we observe that as the depth increases, the perfect fifth appears first (when just a few harmonics are present), followed by the perfect fourth, with other ratios emerging at still greater depth. Furthermore, Theorem \ref{secondtheorem} establishes that at a given depth of a complex tone, these maxima are visible over a finite range of frequencies and disappear beyond a certain threshold.

The topic of consonance has been explored from a modern perspective by W. Sethares in \cite{Sethares1993} and further developed in his insightful book, which leverages contemporary scientific methods \cite{Sethares2008}. For an accessible introduction to the connections between mathematics and music, we recommend the article by C. Nolane \cite{Nolane}.

{\bf Organization of the paper.}  In \S2, we collect the various preliminary concepts. These include the definitions of pure and complex tone and the concept of (continuous) cosine similarity.  
In \S3, we present the numerical experiments showing how this notion essentially recovers the consonance curves of Plomp and Levelt. A testimony to the fact that our $\cosim$ measure is a good quantification of the notion of consonance 
is given by the fact that we do obtain maxima for the values $r=\displaystyle\frac{3}{2}$ (perfect fifth)
and $r=\displaystyle\frac{4}{3}$ (perfect fourth), as well as the other ratios considered consonant in music. 
In section \S4 the main results are presented: it is demonstrated how the consonance measure based on continuous cosine similarity of pure and complex tones is maximized by the frequency ratios considered as consonant in music theory. These results consider the continuous model, independent of starting frequency and sampling rate.
In \S5 we apply our theory to the simplest musical instrument - a single string with is either plucked or hit.
In \S6 we present concluding remarks.

%
%
%
\section{Preliminaries}

In this section we introduce a collection of notions which we consider as `folklore', but which (to the best of our knowledge) have not appeared together before as a set of formal definitions. 

We define a pure tone as a sound with a sinusoidal wave form, that is a sound wave with constant frequency, amplitude and phase-shift.

\medskip 

{\bf Definition 1} 
{\sl A pure tone}  is the following function:
\begin{equation}
 w(f,t) = \sin\left(2\pi f t + \phi\right),   \qquad  0\le t\le T \end{equation}\label{Pure_Tone}
where
\begin{itemize}
\item $t$ is real variable (`time')
\item $f$ is a positive real parameter (`frequency')
\item $\phi$ is a real parameter (`phase shift')
\item $T$ is a positive parameter ('duration of the tone')
\end{itemize}

\medskip

In acoustic applications  $f$ is a frequency from the human audible range, i. e. $20$ Hz --- $20 000$ Hz and the 
phase shift $\phi$ is zero. A graph of a pure tone is displayed in Figure \ref{fig:Tone_Cosim}(a).

\medskip

{\bf Definition 2}
{\sl A harmonic} of a pure tone  with frequency $f$ is a pure tone with  frequency $nf$, where $n$ is a positive integer.

 We define {\sl a complex tone}\footnote{Note that this concept is used in a more restricted sense than in \cite{Roederer2009}} of {\sl depth} $N$, a positive integer, as the algebraic sum of pure tones with frequencies $f, 2f, \dots, N f$. Hence we have:
 \begin{equation}
     w_c(f,N,A,t) = \sum_{n=1}^N a_n\, w(nf,t),    
 \end{equation}\label{Complex_Tone}
 where $A=[a_1,\dots,a_n]$ is a vector with positive entries (`amplitudes of the harmonics').

Often, we take $A=[1,1,\dots,1]$, and omit $A$  as a parameter of $w_c$.
Moreover, when also the depth is equal to one, the second definition reduces to the first.

We show that, in order to observe the phenomena we describe below, one does not need to choose the value of $T$ (`duration') too precisely.
However, the results are most visible if one chooses 
$T$ to be about $4/f$, where $f$ is the lowest frequency present in the complex tone.

\medskip

{\bf Definition 3}  We define  {\sl a mixed tone} with amplitude vector 
$A=[a_1,\dots,a_n]$ and frequency vector $F=[f_1,\dots,f_n]$
as the following sum.
\begin{equation}
    w_m(F,N,A,t) = \sum_{n=1}^N a_n\, w(f_n,t)       \label{Mixed_Tone}
\end{equation}

\medskip

{\sl The cosine similarity} of the vector $X  = [x_0,\dots, x_K]$ and $Y= [y_0,\dots, y_K]$ is given by the following formula:

\[\cosim{(X,Y)} = \frac{\sum_{k=1}^K  x_k y_k}{\sqrt{\sum_{k=1}^K  x_k^2}\,\sqrt{  \sum_{k=1}^K  y_k^2 }}      \]

It follows from the Cauchy--Schwartz inequality that cosine similarity (and its continuous version) takes values in the interval $[-1,1]$.

\begin{example} Consider the vectors
 \[ X=[6,6,3],\qquad Y=[-6,-6,-3],\qquad\textrm{and}\qquad Z=[1,-1,0] \]
 A short computation shows that $\cosim(X,X)=1$, $\cosim(X,Y)=-1$, and $\cosim(X,Z)=0$. 
\end{example}

\textbf{Definition 4.}\label{cosine_similarity}
The continuous cosine similarity of functions $u,v:[0,T]\to \mathbb{R}$ is defined as:
\begin{equation}\label{Cosim}
\cosim{(u,v)} = 
\frac{\int_0^T  u(t) v(t)\,dt}{\sqrt{\int_0^T  u^2(t)\,d t }\sqrt{\int_0^T  v^2(t)\,d t}}      
\end{equation}


\section{Numerical experiments: cosine similarity recovers the Plomp--Levelt model}

In order to conduct numerical experiments, we will consider sampling  of a  pure tone of a certain sampling rate. 
Hence we choose (a large) integer $s$ (the sampling rate) and compute the values of the functions $w$ and $w_c$ at the points
\[      t_i = i*\frac{T}{s}, \qquad\textrm{for}\quad i=0,\dots,s          \]
In the experiments below, we use s= 44000.

Given pure tones $w_1$ and $w_2$, we can form their sample vectors:

\[    X_1  = [w_1(t_0), w_1(t_1),\dots, w_1(t_s)]        \]
\[    X_2  = [w_2(t_0), w_2(t_1),\dots, w_2(t_s)]           \]

By cosine similarity of two pure tones we mean the cosine similarity of the vectors $X_1$ and $X_2$ obtained by sampling these tones at a given rate over a certain interval.

For example, if $f=440$ and $f_1=660$, then their cosine similarity is $cs\approx 0.002$.

On the other hand, if  $f=440$ and $f_1=700$, then their cosine similarity is  $cs\approx 0.039$.

We introduce a frequency ratio parameter $r$ as a real number from the interval $[1,2]$, 
this allows us to focus our study of consonance on two tones of frequency $f$ and $rf$.

Figure \ref{fig:Tone_Cosim}(b) displays the cosine similarity of  $f=440$ Hz and $r*f$, for $r\in [1,2]$

\begin{figure}
    \begin{subfigure}{0.49\textwidth}
	   \begin{center}
            \centering
            \includegraphics[width=\linewidth]{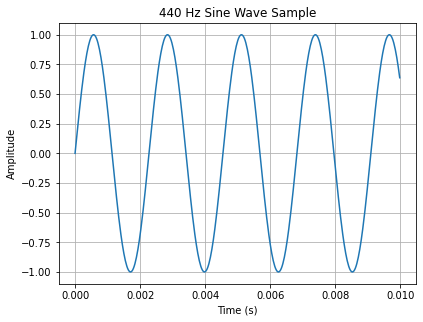}
    \end{center}
    \end{subfigure}
    \begin{subfigure}{0.49\textwidth}
        \begin{center}
            \centering
            \includegraphics[width=\linewidth]{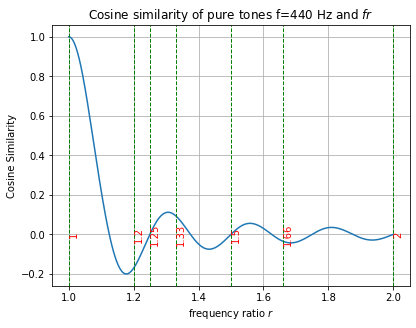}
        \end{center}
    \end{subfigure}
\caption{(a) Left: Graph of a pure tone. (b)~Right: Cosine similarity of a pure tones $f$ and $rf$ for $r$ ranging between one and two.} 
\label{fig:Tone_Cosim}
\end{figure}

 We define the consonance of two discrete complex tones of frequencies $f$ and $g$ and of depth $N$  as the  cosine similarity of the two vectors $X$ and $Y$ discretizing the complex tones   $w_c(f,N,t)$ and $w_c(g,N,t)$ :
\[   \cons (f, g, N) =  \cosim\left(X,Y \right)   \]

 By choosing different amplitude vectors of the harmonics, we model different instruments (see Sethares\cite{Sethares2008}, Chapter 2). However, in our numerical experiments in this section we take all the weights equal to one. We will often consider complex tones of depth $N=6$, which seems to be a reasonable choice for many standard musical instruments. 

 Figure \ref{fig:Cosine_Similarity} displays the consonance of discrete complex tones of frequencies $f$ and $rf$ for depth $N\in\{3,4,5,6\}$.

\begin{figure}
    \begin{subfigure}{0.49\textwidth}
            \centering
            \includegraphics[width=\linewidth]{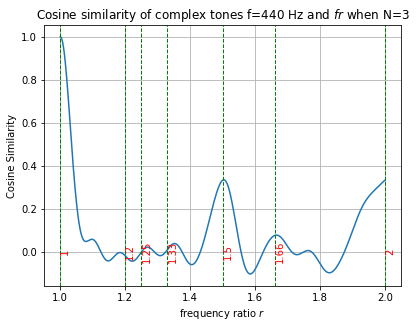}
    \end{subfigure}    
    \begin{subfigure}{0.49\textwidth}
            \centering
            \includegraphics[width=\linewidth]{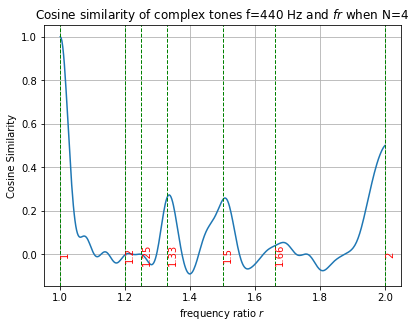}
    \end{subfigure} 
    \begin{subfigure}{0.49\textwidth}
        \begin{center}
            \centering
            \includegraphics[width=\linewidth]{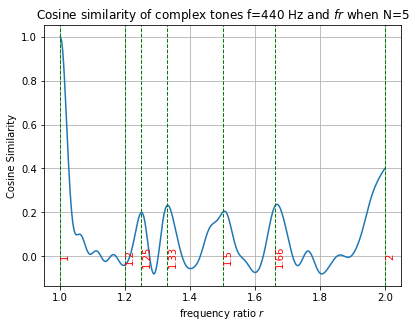}
        \end{center}
    \end{subfigure}    
    \begin{subfigure}{0.49\textwidth}
        \begin{center}
            \centering
            \includegraphics[width=\linewidth]{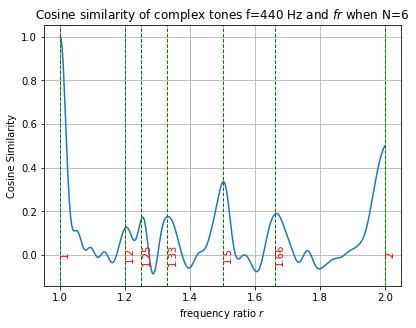}
        \end{center}
    \end{subfigure}    
    \caption{The consonance of samples of complex tones $f$ and $rf$ of depth $N$ ranging from three to six (compare to Figure~\ref{fig:Continuous_Cosine_Similarity} presenting the continuous counterparts in place of finite samples).}
    \label{fig:Cosine_Similarity}
\end{figure}

Our consonance curve for depth $N=6$ is stikingly similar to the consonance curve of Plomp and Levelt displayed in Figure \ref{fig:Plomp--Levelt} obtained by psychoacoustic experimentation.

\begin{figure}
\begin{center}
        \centering
        \includegraphics[width=0.8\linewidth]{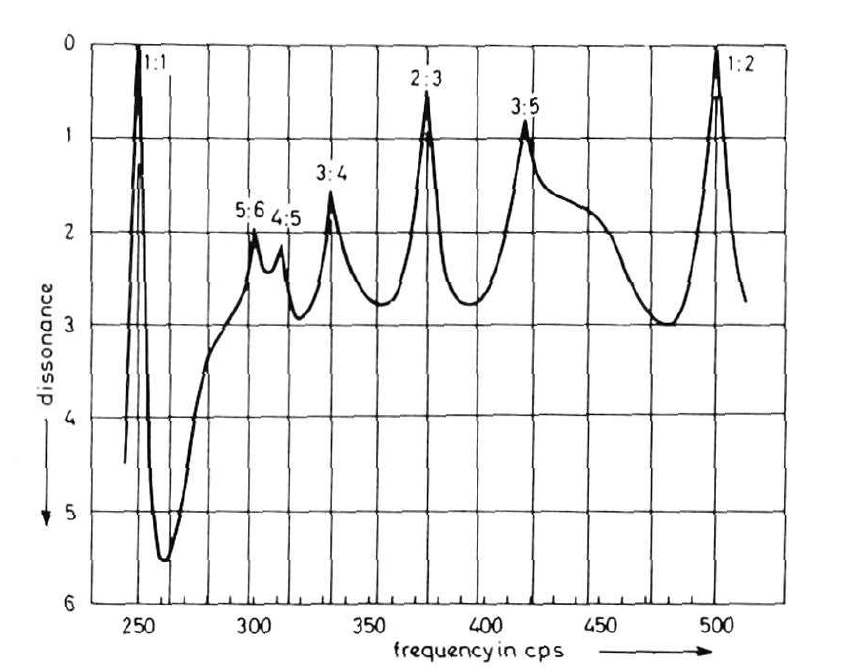}
    \end{center}
\caption{The Plomp--Levelt consonance diagram, taken from~\cite{Plomp65}} 
\label{fig:Plomp--Levelt}
\end{figure}

\section{Analytic Description of consonance of complex tones}	

Recall that we have defined a pure tone $w$  of frequency $f$ and a complex tone of same frequency and depth $N$ by the formulas (see Defintions 1 and 2):
\[ w(f,t) = \sin\left(2\pi f t \right),   \qquad  0\le t\le T    \]
\[    w_c(f,N,A,t) = \sum_{n=1}^N a_n\, w(nf,t),           \]
and decided to omit $A$ from notation if all $a_i$ are equal to one.

We define the consonance of complex tones of depth $N$ and frequencies $f$ and $g$ as their continuous cosine similarity (see formula \ref{cosine_similarity}) 

\begin{equation}\label{eq:paircor}
		\cons(f,g,N) = \cosim \left(w_c(f,N,A,t), w_c(g,N,B,t) \right)
\end{equation}

When the depth $N$ is equal to one, we will omit it from notation.

The results for depth $N$ ranging between three and six and for $f=440$ are shown in Figure~\ref{fig:Continuous_Cosine_Similarity}.

\begin{figure}
    \begin{subfigure}{0.49\textwidth}
            \centering
            \includegraphics[width=\linewidth]{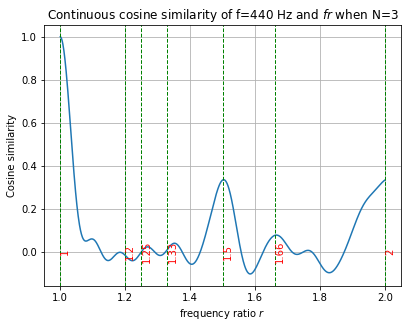}
    \end{subfigure}    
    \begin{subfigure}{0.49\textwidth}
            \centering
            \includegraphics[width=\linewidth]{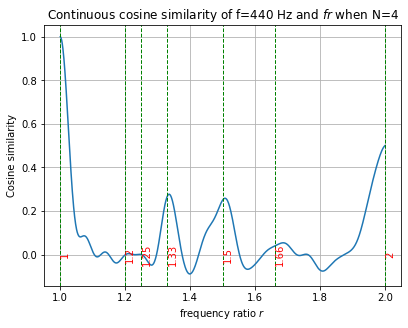}
    \end{subfigure} 
    \begin{subfigure}{0.49\textwidth}
        \begin{center}
            \centering
            \includegraphics[width=\linewidth]{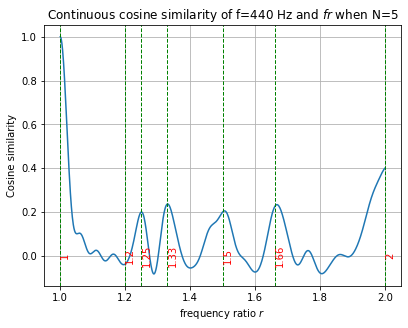}
        \end{center}
    \end{subfigure}    
    \begin{subfigure}{0.49\textwidth}
        \begin{center}
            \centering
            \includegraphics[width=\linewidth]{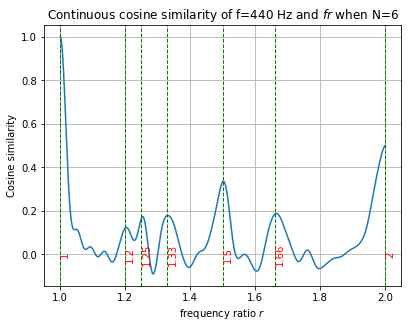}
        \end{center}
    \end{subfigure}    
    \caption{Continuous cosine similarity of complex tones with frequencies $f$ and $rf$ for depth $N$ ranging from three to six. Note the similarity to Figure~\ref{fig:Cosine_Similarity} containing the sampled version of the curves.}
    \label{fig:Continuous_Cosine_Similarity}
\end{figure}


It is clearly visible that the consonance function distinguishes by its maxima the consonant intervals of the minor third (m3) with the frequency ratio in the just intonation equal to $r=6/5$, the major third (M3) with the ratio $r=5/4$, the perfect fourth (P4) and fifth (P5) with the frequency ratios, respectively, $r=4/3$ and $3/2$, and the major sixth (M6) with the frequency ratio $5/3$, in addition to the fundamental ones of the perfect unison (P1) and the perfect octave (P8) with the ratios $1/1$ and $2/1$.

\medskip

To understand the origin of the maxima of the consonance curve, let us analyse in more detail the consonance of the pure tones  $w(nf,t)$ and $w(mrf,t)$, where $r$ is a number from the interval $[1,2]$, i.e.
\begin{equation}\label{eq:paircor_puretones}
		\cons(nf,mrf) = \cosim \left(w(nf,t), w(mrf,t) \right)
\end{equation}

\begin{lemma}
For positive integers $n$ and $m$: 

\begin{equation*}
	   \int_0^T  w(nf,t)w(mrf,t)\,dt = \frac{1}{{4\pi f}}\left( {\frac{{\sin \big( {2\pi f T\left( {rm - n} \right)} \big)}}{{rm - n}} - \frac{{\sin \big( {2\pi f T \left( {rm + n} \right)} \big)}}{{rm + n}}} \right)
\end{equation*}
\begin{equation*}\label{eq:pure_tone_stanard_dev}
\left|\int_0^T  w(nf,t)^2\,dt - \frac{T}{2}\right| \le  \frac{1}{8\pi nf}
\end{equation*}
\end{lemma}

\begin{proof}
The appropriate integrals are evaluated in a standard way.
\begin{equation*}
    \begin{split}
        \int_0^T  w(nf,t)w(mrf,t)\,dt 
        &= \int_0^T  \sin\left(2\pi nft\right)\sin\left(2\pi mrft\right)\,dt \\
        &= \frac{1}{2}\int_0^T  \cos\left(2\pi mrft-2\pi nft\right)
        - \cos\left(2\pi mrft+2\pi nft\right)\,dt \\
        &= \frac{1}{2}\int_0^T  \cos [2\pi f\left( mr- n\right)t]
        - \cos [2\pi f \left( mr+ n\right)t]\,dt \\
        &= \frac{1}{4\pi f} \left\{\frac {\sin [2\pi f\,T\left( mr- n\right)]}{mr-n}
        - \frac{\sin [2\pi f\,T \left( mr+ n\right)]}{mr+n} \right\}
    \end{split}
\end{equation*}

Above, we have used the formula $\sin\phi \sin \theta = \cos(\phi-\theta)-\cos(\phi+\theta)$.

For the second inequality, notice that

\begin{equation*}\label{eq:pure_tone_stanard_dev}
    \begin{split}
        \int_0^T  w(nf,t)^2\,dt 
        &= \frac{1}{2\pi n f} \int_0^{2\pi nf\,T} \sin^2u\,du\\
        &= \frac{1}{2\pi n f}  \left(\frac{u}{2}-\frac{1}{4}\sin 2u\right) \bigg|_0^{2\pi nf\,T}\\
        &= \frac{T}{2} + \frac{1}{8\pi nf} \sin(4\pi n  f\, T)
     \end{split}
\end{equation*}

\end{proof}

\begin{corollary}
For positive integers $m\le n$: 
\begin{equation*}\label{eq:paircor1}
	\left|\cons(nf,mrf) - \frac{1}{{2\pi f\,T}}\left( {\frac{{\sin \big( {2\pi f\, T\left( {rm - n} \right)} \big)}}{{rm - n}} - \frac{{\sin \big( {2\pi f\, T \left( {rm + n} \right)} \big)}}{{rm + n}}} \right)\right|
    =\O\left(\frac{1}{nf}\right)
\end{equation*}
\end{corollary}

\begin{proof} Starting from the definition, we have
\[  \cons(nf,mrf) = \cosim(w(nf,t),w(mrf,t)) = \frac{\int_0^T  w(nf,t)w(mrf,t)\,dt}{\sqrt{\int_0^T  w(f,t)^2\,dt}\sqrt{\int_0^T  w(rf,t)^2\,dt}} 
\]
The result follows from a routine calculation involving the inequality $1-x\le \frac{1}{1+x}$.
\end{proof}

It follows that when looking for maxima (and minima) of the correlation function and $nf>>0$, we can focus on the function $AC$ (asymptotic consonance) defined as follows.

\begin{equation}\label{eq:Asymptotic_Consonance}
AC(nf, mrf) = \frac{1}{{2\pi f\,T}}\left( {\frac{{\sin \big( {2\pi f\, T\left( {rm - n} \right)} \big)}}{{rm - n}} - \frac{{\sin \big( {2\pi f\, T \left( {rm + n} \right)} \big)}}{{rm + n}}} \right)
\end{equation}

For fixed $f,T$ and $m\le n$, as a function of $r$, the function 
$AC$ is a linear combination of two sinusoidal functions with arguments 
\[2\pi f\, T\left( {rm - n}\right) \qquad\textrm{and}\qquad 2\pi f\, T\left( {rm + n} \right) \] 
and amplitudes proportional to $1/(rm-n)$ and $1/(rm+n)$. As long as $rm>n$, both denominators are strictly positive and the function takes the form of regular oscillations with decreasing amplitude (a consequence of the linear dependence on $r$ in the denominators). See the example run in Fig.\ref{fig:s44035_and_s44053}(a).

\begin{figure}
    \begin{subfigure}{0.49\textwidth}
	   \centering
	   \includegraphics[width=\linewidth,height=0.25\textheight]{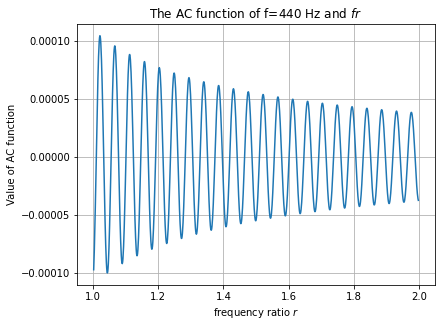}   
    \end{subfigure}
    \begin{subfigure}{0.49\textwidth}
	   \centering
	   \includegraphics[width=\linewidth,height=0.25\textheight]{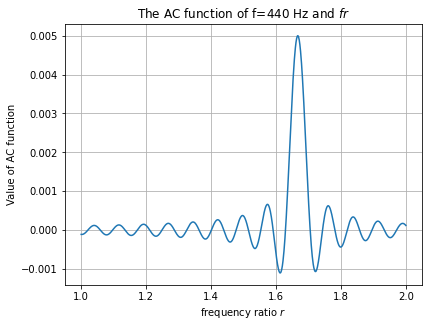}
    \end{subfigure}    
    \caption{  The function $AC$ for $f=440$ Hz and \newline
    (a) $n=3, m=5$ and (b)  $n=5, m=3$. }
    	\label{fig:s44035_and_s44053}
\end{figure}
The behavior changes when the denominator $rm-n$ can be zero (possible only when  $m \le n$). Recall that $r$ satisfies $1\le r\le 2$ (one octave). The first component in the formula (\ref{eq:Asymptotic_Consonance}) becomes  singular (with a singularity of type $\mathop {\lim }\limits_{x \to 0} \left( {\sin (x)/x} \right))$, which results in a finite value, but one that far exceeds the amplitude of regular oscillations for $rm\ne n$ such as in Fig.\ref{fig:s44035_and_s44053} (a). An example is shown in Fig.\ref{fig:s44035_and_s44053}(b).

We have arrived at the following theorem.

\begin{thm} \label{maintheorem}
Let $f$ be a positive number and $1\le r \le 2$. The local maxima of the consonance function $\cons(f,rf)$  (formula \ref{eq:paircor} and figure \ref{fig:Continuous_Cosine_Similarity}) are a sum of the individual ``singular" (i.e. with $m\le n$) pair consonances (\ref{eq:paircor_puretones}). They occur at $r=n/m$. For $1\le n, m \le 6$, we have the following list of maxima: 
\medskip
\begin{enumerate}
	\item $r=1/1=2/2=3/3=4/4=5/5=6/6$ (perfect unison P1)
	\item $r=6/5$ (minor third m3)
	\item $r=5/4$ (major third M3)
	\item $r=4/3$ (perfect fourth P4)
	\item $r=3/2 = 6/4$ (perfect fifth P5)
	\item $r=5/3$ (major sixth M6)
	\item $r=2=2/1=4/2=6/3$ (perfect octave P8)
\end{enumerate}

\end{thm}

\medskip
The above result identifies all the classical consonances except the minor sixth (m6) (see \cite{Roederer2009}, p. 171). It is also not visible on an experimentally motivated curve\footnote{The curve was generated using the programs designed by W. A. Sethares \cite{Sethares2008}} of Plomp and Levelt \cite{Plomp65} (see Fig.\refeq{fig:e44066}).
\begin{figure}[ht]
	\centering	\includegraphics[width=0.75\linewidth,height=0.3\textheight]{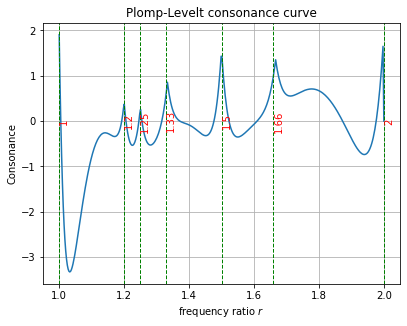}
	\caption{Experimentally motivated Plomp-Levelt consonance curve for $f=440$ Hz, $N=M=6$.}
	\label{fig:e44066}
\end{figure}
The minor sixth interval corresponds to the irreducible fraction $8/5$, which according to our reasoning requires correlations of, at least, eight harmonics rarely significantly present in real musical timbres.

Next, let's turn our attention to complex tones of depth $N\ge 1$, i.e.

\[    w_c(f,N,A,t) = \sum_{n=1}^N a_n\, w(nf,t),        \]

The next lemma shows that the denominators present in the cosine similarity formula play little role in our considerations under appropriate assumptions on $f$ and $N$.

\begin{lemma}\label{bound_on_wc2} Let $a_n\in[0,1]$ for $n=1,\dots,N$.
We have
\[\left|\int_0^T w_c(f,N,A,t)^2\, dt - \frac{NT}{2}\right| \le \frac{N^2}{2\pi f}
\]
\end{lemma}
\begin{proof}
We are looking for a bound on the integral of $w_c^2$,
that is the following expression
\[ \sum_{n=1}^N \sum_{m=1}^N \int_0^T a_n a_m w(nf,t) w(mf,t)\,dt  \]
We obtain the desired bound by looking at the diagonal (i.e. those, where $n=m$) and the non-diagonal entries in the summation, and applying the first and second part of the earlier lemma.
For the diagonal entries, we have
\[\frac{NT}{2} - \frac{N}{8\pi f} \le \sum_{n=1}^N  \int_0^T a^2_n  w^2(nf,t) \,dt \le \frac{NT}{2}+ \frac{N}{8\pi f} \]

For the non-diagonal elements, we have
\[ -\frac{N(N-1)}{2\pi f} \le \sum_{n,m=1,n\neq m}^N \int_0^T a_n a_m w(nf,t) w(mf,t)\,dt \le \frac{N(N-1)}{2\pi f} \]
By adding these two sets on inequalities we obtain the bound in the statement of the lemma.
\end{proof}

From this lemma we conclude that when $N^2<<2\pi f$,
which does hold in the cases we consider, the product of square roots which appears in the cosine similarity formula is well approximated by $\frac{NT}{2}$.
Hence the behaviour of the $\cosim$ function is determined by the numerator. From the linearity of the integral and distributivity of multiplication with respect to addition we conclude that the intergral 
\[  \int_0^T w_c(f,t) w_c(rf,t)\, dt \]
is the linear combination of the terms
\[  \int_0^T w(nf,t) w(mf, t)\, dt \]
we have considered earlier. 

This establishes the following result.

\begin{thm}\label{secondtheorem}
Consider two complex tones of frequencies $f$ and $rf$, where $1\le r\le 2$.
\begin{itemize}
    \item As their depth $N$ increases to infinity, so does the number of maxima of the asymptotic consonance function.
    \item If their depth $N$ is assumed to be fixed, as $f$ increases to infinity,  the maxima of the consonance function disappear. 
\end{itemize}
\end{thm}

In the earlier numerical experiments, we assumed that all the harmonics had the same amplitude. For the majority of musical instruments, higher harmonics have lower amplitudes. It is not hard to describe the situation in which the amplitude of each successive harmonics decreases by a constant factor $d\in(0,1)$ relative to the previous one. 
Hence we consider a complex tone of the following form:
\[  w_c(f,N,A,t)=\sum_{n=1}^N d^n \sin(2\pi n f t)   \]

Hence the numerator of the expression giving the $\cosim$ function of 
$w_c(f,N,A,t)$  and $w_c(rf,N,A,t)$ is given by

\[ \int_0^T w_c(f,N,A,t)\cdot w_c(r f,N,A,t) \,dt = 
\sum\limits_{n = 1}^N \sum\limits_{m = 1}^N {d^{(n+m)}\int_0^T  w(nf,t)w(mrf,t)\,dt }
\]
The integral on the right hand side has been computed in Lemma 1. By Lemma \ref{bound_on_wc2} we know that the rate of growth of the integrals appearing in the denominator of the formula for cosine similarity (Definition 4) as a function of depth $N$ is linear and hence the pattern of local maxima of the consonance function observed earlier persists to the situation with harmonics of decaying amplitude.  

Figure \ref{fig:440d08_and_pl440d08}(a) shows the result for $d=0.8$. As before, clear maxima occur for m3, M3, P4, P6, M6 and P8, although the relative magnitudes of the maxima have changed. The corresponding Plomp--Levelt--type curve is shown in Figure \ref{fig:440d08_and_pl440d08}(b).

\begin{figure}
    \begin{subfigure}{0.49\textwidth}
	   \centering
	   \includegraphics[width=\linewidth,height=0.3\textheight]{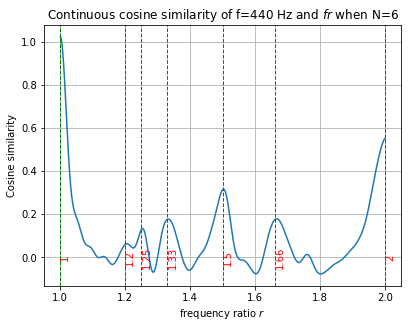}   
    \end{subfigure}
    \begin{subfigure}{0.49\textwidth}
	   \centering
	   \includegraphics[width=\linewidth,height=0.3\textheight]{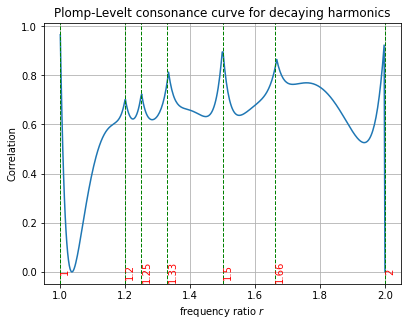}
    \end{subfigure}
    \caption{ Frequence $f=440$, depth $N=6$, decay rate $d=0.8$ (a) Consonance based on correlation
    (b) Plomp-Level consonance curve}
	   \label{fig:440d08_and_pl440d08}
\end{figure}


\section{Concluding remarks}

The presented results show that the mathematical notion of correlation is a good
good quantification of the musical notion of consonance: consonant musical intervals are those whose frequency ratios give local maxima of the consonance function. These appear only once the harmonics of a pure tone are included, and the number of these maxima increases with the number of harmonics present (depth).  A testimony to this is contained in the graphs
presented above, which have maxima at   $r=\displaystyle\frac{3}{2}$ (perfect fifth), $r=\displaystyle\frac{4}{3}$ (perfect fourth), 
$r=\displaystyle\frac{5}{3}$ (major sixth),
$r=\displaystyle\frac{5}{4}$ (major third),
as well as some other ratios which are considered 
consonant in music theory.

The significance of this contribution stems from the fact that this description of consonance is a purely analytic property of the mathematical models of sound,
and does not appeal in any way to physiological or psycho-acoustic phenomena.

\section*{Acknowledgement}
P.D. acknowledge the support of Dioscuri program initiated by the Max Planck Society, jointly managed with the National Science Centre (Poland), and mutually funded by the Polish Ministry of Science and Higher Education and the German Federal Ministry of Education and Research.



\begin{thebibliography}{5}


\bibitem{fletcher12} Fletcher, Neville H., and Thomas D. Rossing, \textit{The Physics of Musical Instruments}. Springer Science \& Business Media, 2012.

\bibitem{hall86} Hall, Donald E., 
\textit{Piano string excitation in the case of small hammer mass }
J. Acoust. Soc. Am., 79(1), 141--147, 1986.


\bibitem{Nolane} C. Nolane, {\sl Mathematics and Music},
in ``The Pricenton Companion To Mathematics'', T. Gowers, ed. Priceton Univ. Press, 2008

\bibitem{nurowski04} Nurowski, Pawe\l,  \textit{Why does a piano sounds different than a harpsichord?}, unpublished.	

\bibitem{Plomp65} R. Plomp and W.~J.~M. Levelt,{\sl  Tonal Consonance and Critical Bandwidth}, Journal of the Acoustical Society of America,  38/1965: 548--560.

\bibitem{Roederer2009} J.G. Roederer, Introduction to the physics and psychophysics of Music, Spriner 2007.

\bibitem{Sethares1993} W.A. Sethares, Local consonance and the relationship between timbre and scale, J. Acounst. Soc. Am 94(3), 1993.

\bibitem{Sethares2008} W.A. Sethares, Tuning, Timbre, Sepctrum and Scale, 2nd Ed., Springer 2008.

\end{thebibliography}
\end{document}